\documentclass{article}
\usepackage[T1]{fontenc}
\usepackage[latin9]{inputenc}
\usepackage{amsmath,amsthm} 
\usepackage{graphicx}
\usepackage{amssymb}
\usepackage{enumitem}
\usepackage{natbib} 
\usepackage[letterpaper]{geometry} 
\usepackage{tikz}
\usetikzlibrary{patterns,automata,arrows,shapes,snakes,topaths,trees,backgrounds,positioning,through,calc,arrows.meta}
\usepackage{authblk}
\usepackage{setspace}
\usepackage{multicol}
\usepackage{stmaryrd}  
\usepackage{comment}
\usepackage{hyperref}
\hypersetup{colorlinks=true, citecolor=blue, linkcolor=blue}  

\usepackage{xcolor}
\definecolor{medgreen}{rgb}{0.0, 0.75, 0.0}

\usepackage[capitalize]{cleveref}

\usepackage{mathrsfs}

\theoremstyle{definition}
\newtheorem{theorem}{Theorem}[section]

\newtheorem{proposition}[theorem]{Proposition} 
\newtheorem{corollary}[theorem]{Corollary}
\newtheorem{definition}[theorem]{Definition}

\newtheorem{fact}[theorem]{Fact}

\newtheorem{remark}[theorem]{Remark}  
\newtheorem{problem}[theorem]{Problem}

\setcitestyle{round}

\begin{document} 

\title{A new proof of Funayama's theorem}

\author[1]{Guram Bezhanishvili}
\author[2]{Wesley H. Holliday}
\affil[1]{\small New Mexico State University}
\affil[2]{\small University of California, Berkeley}

\date{March 24, 2026}

\maketitle

\renewcommand{\thefootnote}{}
\renewcommand{\thefootnote}{\arabic{footnote}}

\begin{abstract} Funayama proved that a lattice embeds into a complete Boolean algebra in such a way that all existing joins and meets are preserved if and only if the lattice satisfies the join-infinite and meet-infinite distributive laws. There are several proofs of this classic result in the literature. In this note, we provide a new and purely order-theoretic proof of Funayama's theorem, as well as of generalizations of the theorem.
\end{abstract}

\section{Introduction}\label{Intro}

It is a consequence of classic results of Birkhoff \citeyearpar{Birkhoff1937} and Stone \citeyearpar{Stone1938} that for any distributive lattice $L$, there is an embedding of $L$ into a complete Boolean algebra that preserves finite meets and joins, but not necessarily infinite meets and joins. The question naturally arose of when there is a \textit{complete embedding} of $L$ into a complete Boolean algebra, that is, an embedding that preserves all existing meets and joins. Not every distributive lattice can be completely embedded into a complete Boolean algebra since, as observed by von Neumann \citeyearpar{Vonneumann1936}, complete Boolean algebras satisfy the join-infinite and meet-infinite distributive laws:

\begin{definition}
A lattice $L$ satisfies \textit{join-infinite distributivity} (JID) if for each existing join $\bigvee S$ and $a\in L$, the join $\bigvee \{a \wedge s \mid s\in S\}$ also exists and $$a \wedge \bigvee S = \bigvee \{a \wedge s \mid s\in S\}.$$ 
\textit{Meet-infinite distributivity} (MID) is defined dually.     
\end{definition}
\noindent 

While not all distributive lattices satisfy JID or MID, Funayama \citeyearpar[Thm.~6]{Funayama1959} proved that satisfying these principles is sufficient for the existence of a complete embedding into a complete Boolean algebra. 

\begin{theorem}[Funayama 1959]\label{Funayama} For any  lattice $L$, there is a complete embedding of $L$ into a complete Boolean algebra iff $L$ satisfies JID and MID. 
\end{theorem}

\noindent Funayama's proof is complicated. Gr\"{a}tzer \citeyearpar[p.~109]{Gratzer1971} gave a more accessible proof of \cref{Funayama} under the assumption that $L$ is complete, though the assumption of completeness can be dropped (see \citealt{Bezhanishvili2013}). The basic idea of the proof is that the Boolean envelope $\mathscr{B}L$ of $L$ preserves all existing joins in $L$ iff $L$ satisfies JID, and it preserves all existing meets iff $L$ satisfies MID. Taking the MacNeille completion $\mathfrak{M}\mathscr{B}L$ of $\mathscr{B}L$ then completes the proof since the MacNeille completion  preserves all existing joins and meets. 

Inspection of Gr\"{a}tzer's proof shows that JID (resp.~MID) is equivalent to the existence of a $(\wedge,\bigvee)$-embedding (resp.~$(\vee,\bigwedge)$-embedding), i.e., an embedding that preserves all finite meets and existing joins (resp.~all finite joins and existing meets), into a complete Boolean algebra. Thus, Funayama's theorem can be rephrased as follows.

\begin{theorem}\label{JIDMID} For any lattice $L$:
\begin{enumerate} 
\item\label{JIDMID1} there is a $(\wedge,\bigvee)$-embedding of $L$ into a complete Boolean algebra iff $L$ satisfies JID; 
\item\label{JIDMID2} there is a $(\vee,\bigwedge)$-embedding of $L$ into a complete Boolean algebra iff $L$ satisfies MID;
\item\label{JIDMID3} there is a complete embedding of $L$ into a complete Boolean algebra iff $L$ satisfies JID and MID.
\end{enumerate}
\end{theorem}

In pointfree topology, the main objects of interest are complete lattices satisfying JID, called \textit{frames} or \emph{locales} (see, e.g.,  \citealt{Picado2012}). Assuming $L$ is complete, the right-to-left direction of \cref{JIDMID}.\ref{JIDMID1} says that every frame is isomorphic to a subframe of a Boolean frame. Isbell~\citeyearpar{Isbell1972} gave an alternative proof of this result by taking the assembly frame $\mathscr{N} L$ of $L$, then considering its Booleanization $\mathfrak {B} \mathscr{N} L$, and finally observing that the $(\wedge,\bigvee)$-embedding $L \to \mathscr{N} L$ factors through $\mathfrak{B} \mathscr{N} L$ (also see \citealt{Johnstone1982}, p.~53).

While the proofs of Gr\"{a}tzer and Isbell are different, if $L$ is complete, then $\mathfrak{M}\mathscr{B}L$ and $\mathfrak{B} \mathscr{N} L$ are isomorphic (\citealt[Thm.~4.13]{Bezhanishvili2013}). The proof of the isomorphism uses Priestley duality, which yields another direct proof of Funayama's theorem (\citealt[Remark 4.7]{Bezhanishvili2013}).

In this paper, we provide yet another proof of Funayama's theorem. The main features of this new proof are the following:
\begin{itemize}
\item The proof is purely order-theoretic. Inspired by the representation of complete lattices by Allwein and MacCaull \citeyearpar{Allwein2001}, we construct from a given poset $P$ a partial ordering of certain pairs of elements from $P$ and then take the regular open algebra of the poset of pairs, which is a Boolean algebra by a classic result of Tarski \citeyearpar{Tarski1937}.
\item Our approach allows generalizing \cref{JIDMID}.\ref{JIDMID1} to arbitrary meet semilattices and \cref{JIDMID}.\ref{JIDMID2} to arbitrary join semilattices, using the concept of exact joins and meets (see Definition~\ref{ExactDef}), as well as establishing Funayama's original theorem.
\item The proof does not require the use of Boolean envelopes and MacNeille completions (Gr\"{a}tzer) or frame assemblies and Booleanizations (Isbell). It is also choice-free, unlike the proof using Priestley duality, which depends on the Prime Ideal Theorem. 
\item On the other hand, the complete Boolean algebra we construct is isomorphic to that of Gr\"{a}tzer if we start with a distributive lattice, as well as that of Isbell if we start with a frame.
\item Our investigation raises two natural problems (Problems \ref{1stProblem} and \ref{2ndProblem}) that we leave for future work.
\end{itemize}

\section{Construction}

Before introducing our main construction, we recall the notion of {\em distributive} joins and meets (\citealt{MacNeille1937}). We follow Ball \citeyearpar{Ball1984} in calling these {\em exact}, and we note that Bruns and Lakser \citeyearpar{BrunsLakser1970} call a set $S$  \textit{admissible} if it has an exact join.

\begin{definition}\label{ExactDef}
A subset $S$ of a meet semilattice $M$ \textit{has an exact join} if $\bigvee S$ exists in $M$ and for each $a\in M$, $\bigvee \{a\wedge s\mid s \in S\}$ exists in $M$ and $$a \wedge \bigvee S = \bigvee \{a\wedge s\mid s \in S\}.$$ 
Dually, a subset $S$ of a join semilattice $J$ \textit{has an exact meet} if $\bigwedge S$ exists in $J$ and for each $a\in J$, ${\bigwedge \{a\vee s\mid s \in S\}}$ exists in $J$ and $$a \vee \bigwedge S = \bigwedge \{a\vee s\mid s \in S\}.$$ 
Given meet semilattices $M,M'$ and  an embedding $e:M \to M'$, we say that $e$ \textit{preserves all existing exact joins} if for any subset $S\subseteq M$, if $S$ has an exact join, then \[e\left(\bigvee S\right)= \bigvee \{e(s)\mid s\in S\}.\] 
For an embedding $e$ of join semilattices to \textit{preserve all existing exact meets} is defined dually.\end{definition}

We now prove the first of our main results.

\begin{theorem}[Funayama for meet semilattices] \label{FMS} For any poset $P$, there is an embedding $e$ of $P$ into a complete Boolean algebra such that:
\begin{enumerate}
    \item\label{MainThm1} $e$ preserves all existing finite meets;
\item\label{MainThm2} if $P$ is a meet semilattice, then $e$ preserves all existing exact joins.
\end{enumerate}
\end{theorem}

\begin{proof} We may assume that $P$ has more than one element, since otherwise it is the one-element Boolean algebra, in which case we can take $e$ to be the identity embedding, or it is the empty poset, in which case we can take $e$ to be the empty embedding into the one-element Boolean algebra. In addition, since each poset can be embedded into a bounded poset by adjoining a top and/or bottom if necessary, which preserves all existing meets and joins, as well as the exactness of meets and joins, we may assume that $P$ is bounded. We denote its minimum element by $0$ and its maximum element by $1$.

Where $\leq$ is the order of $P$, define $(X_P,\sqsubseteq)$ as follows:
\begin{enumerate}
\item $X_P=\{(a,b)\in P^2\mid a\not\leq b\}$;
\item $(a,b)\sqsubseteq (c,d)$ iff $a\leq c$ and $b\geq d$.
\end{enumerate}
Note that $\sqsubseteq$ is a partial order whose top element is $(1,0)$. Given $x\in X_P$, let \[ {\Downarrow}x=\{y\in X_P\mid y\sqsubseteq x\}\mbox{ and }\mathcal{D}(X_P)=\{U\subseteq X_P\mid \mbox{ for all }x\in U, \mathord{\Downarrow}x\subseteq U\}.\] 
Note that $\mathcal{D}(X_P)$ is the downset topology on $X_P$. The interior and closure operations are defined for $U\subseteq X_P$ by 
\[
 \Box U = \{x\in X_P \mid  {\Downarrow} x\subseteq U\} \mbox{ and } \Diamond U = \{x\in X_P \mid  {\Downarrow} x\cap U\neq \varnothing\},
\]
respectively. Hence the regular open sets of the topology are the fixpoints of the $\Box\Diamond$ operation, where $\Box\Diamond U = \{x\in X_P\mid  {\Downarrow}x \subseteq  \Diamond U \}$.  We denote the algebra of regular opens by 
\[\mathbf{B}P = (\{\Box\Diamond U\mid U\in \mathcal{D}(X_P)\}, \subseteq )\]
and  note that it is a complete Boolean algebra by a result of Tarski \citeyearpar{Tarski1937} (see, e.g., \citealt[Ch.~10]{Givant2009}). Recall that in $\mathbf{B}P$, meets are intersections, and joins are given by $\bigsqcup \{U_i\mid i\in I\}=\Box\Diamond \bigcup \{U_i\mid i\in I\}$.
 
Now consider the map $e : P \to \mathbf{B}P$ defined by 
\[e(a) = \begin{cases} \Box\Diamond  {\Downarrow} (a,0) & \mbox{if }a\neq 0 \\ \varnothing &\mbox{otherwise}\end{cases}.\]
 We will prove that $e$ is the desired embedding of $P$ into $\mathbf{B}P$. If $P$ was obtained from an original poset $P^{-}$ by adjoining a new top and/or bottom, then the restriction of $e$ to $P^{-}$ is the desired embedding of $P^{-}$ into a complete Boolean algebra.

First, we prove that $e$ is a poset embedding. To see that $e$ is order-preserving, suppose $a\leq b$. If $a=0$, then $e(a)=\varnothing\subseteq e(b)$. Otherwise, from $a\leq b$ we have $(a,0)\sqsubseteq (b,0)$ and hence $\Box\Diamond \mathord{\Downarrow}(a,0)\subseteq \Box\Diamond \mathord{\Downarrow}(b,0)$, so $e(a)\subseteq e(b)$. Next, we prove that $e$ is order-reflecting. Indeed, if $a \not \leq b$, then $(a,b) \in X_P$ and $(a,b) \in \Box\Diamond{\Downarrow}(a,0)$ since for each $(x,y)\in X_P$, if $(x,y) \sqsubseteq (a,b)$, then $x\leq a$ and $y\geq b$, so $x \leq a$ and $y \geq 0$, and hence $(x,y) \sqsubseteq (a,0)$. On the other hand, we claim that $(a,b)\not\in e(b)$. If $b=0$, this is immediate since $e(0)=\varnothing$, so suppose $b\neq 0$. We must show $(a,b) \notin \Box\Diamond{\Downarrow}(b,0)$, i.e., ${\Downarrow}(a,b) \not\subseteq  \Diamond {\Downarrow}(b,0)$. In particular,  $(a,b) \in {\Downarrow}(a,b)$, since $(a,b) \sqsubseteq (a,b)$, but we claim that $(a,b) \notin  \Diamond {\Downarrow}(b,0)$. For if $(x,y) \sqsubseteq (a,b)$, then $x \leq a$ and $y \geq b$. Thus, $x\not\leq b$, for otherwise $x\leq y$, contradicting $(x,y)\in X_P$. Hence  $(x,y)\not\sqsubseteq (b,0)$.

Now we prove for part \ref{MainThm1} that $e$ preserves finite meets. First observe that 
\[
e(1) = \Box\Diamond  {\Downarrow}(1,0) = \Box\Diamond X_P = X_P.
\]
For nonempty finite meets, suppose $a_1\wedge \dots \wedge a_n\neq 0$. Then by definition, $e(a_1\wedge \dots \wedge a_n) = \Box\Diamond  {\Downarrow}(a_1\wedge \dots \wedge a_n,0)$. We claim that $$\Box\Diamond  {\Downarrow}(a_1\wedge \dots \wedge a_n,0) = \Box\Diamond  {\Downarrow} (a_1,0)\cap \dots \cap \Box\Diamond  {\Downarrow}(a_n,0).$$ First, where $\mathord{\downarrow}x=\{y\in P\mid y\leq x\}$ for $x\in P$, observe that since $ {\downarrow}(a_1\wedge\dots\wedge a_n)= {\downarrow}a_1\cap \dots \cap {\downarrow}a_n$, it follows that $ {\Downarrow}(a_1\wedge \dots\wedge a_n,0) =  {\Downarrow}(a_1,0)\cap\dots\cap {\Downarrow}(a_n,0)$. Then since $\Box\Diamond$ distributes over finite intersections,  $$\Box\Diamond {\Downarrow}(a_1\wedge \dots \wedge a_n,0) = \Box\Diamond(  {\Downarrow}(a_1,0)\cap\dots \cap  {\Downarrow}(a_n,0)) = \Box\Diamond  {\Downarrow}(a_1,0) \cap \dots \cap \Box\Diamond  {\Downarrow}(a_n,0).$$ Next, suppose $a_1\wedge \dots \wedge a_n=0$, so $e(a_1\wedge \dots \wedge a_n)=\varnothing$. We show that $e(a_1) \cap \dots \cap e(a_n) = \varnothing$. Indeed, if $a_i=0$ for some $i$, then we are done by the definition of $e$, so suppose $a_i\neq 0$ for each $i$. Then since $a_1\wedge \dots \wedge a_n =0$, we have $ {\Downarrow}(a_1,0)\cap \dots \cap  {\Downarrow}(a_n,0)=\varnothing$, so 
\[
\Box\Diamond  {\Downarrow}(a_1,0) \cap \dots\cap \Box\Diamond  {\Downarrow}(a_n,0) =\Box\Diamond ( {\Downarrow}(a_1,0)\cap \dots \cap {\Downarrow}(a_n,0)) = \Box\Diamond \varnothing =\varnothing.
\]

For part \ref{MainThm2}, we assume that $P$ is a meet semilattice. To show that $e$ preserves all exact joins, consider any $S\subseteq P$ that has an exact join. If $\bigvee S=0$, then either $S=\varnothing$ or $S=\{0\}$. In either case, $\bigsqcup \{e(s)\mid s\in S\}=\varnothing$, since $e(0)=\varnothing$. Hence $e(\bigvee S)=e(0)=\varnothing=\bigsqcup \{e(s)\mid s\in S\}$. Suppose $\bigvee S\neq 0$. Without loss of generality, we can assume $0\not\in S$. Then to show that $e(\bigvee S)=\bigsqcup \{e(s)\mid s\in S\}$, we must show
\[\Box\Diamond {\Downarrow}\left(\bigvee S,0\right) = \bigsqcup \{\Box\Diamond{\Downarrow}(s,0) \mid s\in S\}.\] By definition of $\bigsqcup$ in $\mathbf{B}P$, we have

\[\bigsqcup \{\Box\Diamond{\Downarrow}(s,0) \mid s\in S\} = \Box\Diamond\bigcup \{\Box\Diamond{\Downarrow}(s,0) \mid s\in S\} = \Box\bigcup \{\Diamond\Box\Diamond{\Downarrow}(s,0) \mid s\in S\}.\]
Thus, it suffices to show
\begin{equation}
\Box\Diamond {\Downarrow}\left(\bigvee S,0\right) = \Box\bigcup \{\Diamond\Box\Diamond{\Downarrow}(s,0) \mid s\in S\}.\label{KeyEq}\end{equation}
The right-to-left inclusion is straightforward: for each $s\in S$, we have $(s,0)\sqsubseteq (\bigvee S,0)$ and thus
$\Box\Diamond {\Downarrow} (s,0)\subseteq \Box\Diamond {\Downarrow} (\bigvee S,0) $. Since this holds for all $s\in S$, we have $\bigcup\{\Box\Diamond {\Downarrow} (s,0) \mid s\in S\}\subseteq \Box\Diamond {\Downarrow} \left(\bigvee S,0\right)$, which implies $\Box\Diamond\bigcup\{\Box\Diamond {\Downarrow} (s,0) \mid s\in S\}\subseteq \Box\Diamond\Box\Diamond {\Downarrow} \left(\bigvee S,0\right)$ and hence 
   $\Box\bigcup \{\Diamond\Box\Diamond {\Downarrow} (s,0) \mid s\in S\}\subseteq \Box\Diamond {\Downarrow} \left(\bigvee S,0\right)$.

For the left-to-right inclusion in (\ref{KeyEq}), suppose $(a,b)$ is in the left-hand side. To show that $(a,b)$ is in the right-hand side of (\ref{KeyEq}), consider any $(a',b')\sqsubseteq (a,b)$. Then since  $(a,b)\in \Box\Diamond {\Downarrow} (\bigvee S,0)$, there is an $(a'',b'')\sqsubseteq (a',b')$ such that $(a'',b'')\in {\Downarrow} (\bigvee S,0)$, i.e., $a''\leq \bigvee S$. Since $\bigvee S$ is an exact join, we have $a''=a''\wedge\bigvee S= \bigvee_{s\in S} (a''\wedge s)$. Since $a''\not\leq b''$, it follows that for some $s\in S$, we have $a''\wedge s\not\leq b''$. Now we claim that  $(a''\wedge s, b'')\in \Box\Diamond{\Downarrow}(s,0)$. If not, then there is an $(a_s,b_s)\sqsubseteq (a''\wedge s, b'')$ such that ${\Downarrow}(a_s,b_s)\cap {\Downarrow}(s,0)=\varnothing$. Thus, $(a_s\wedge s, b_s )\not\in X_P$, which implies $a_s\wedge s\leq b_s$. Since  $(a_s,b_s)\sqsubseteq (a''\wedge s, b'')$, we have $a_s\leq a''\wedge s$ and hence $a_s\leq s$, which with $a_s\wedge s\leq b_s$ implies $a_s\leq b_s$, a contradiction. Hence $(a''\wedge s, b'')\in \Box\Diamond{\Downarrow}(s,0)$, which with $(a''\wedge s, b'')\sqsubseteq (a',b')$ implies $(a',b')\in \Diamond\Box\Diamond{\Downarrow}(s,0)$. Therefore, $(a,b)$ is in the right-hand side of (\ref{KeyEq}).\end{proof}

\begin{remark} As noted in the introduction, the construction of $(X_P,\sqsubseteq)$ in the proof of Theorem \ref{FMS} is inspired by the representation of complete lattices by Allwein and MacCaull~\citeyearpar[p.~207]{Allwein2001}; also see Theorem~2.7 of~\citealt{Holliday2021} and Definition~2.5 of~\citealt{Massas2023}. This representation was specialized to the case of frames in Theorem~4.32 of~\citealt{Bezhanishvili2019}, which was  inspired by the semantics for lax logic developed by Fairtlough and Mendler~\citeyearpar{Mendler1997}. The relation $\sqsubseteq$ is precisely the relation $\geq_2$ in Theorem~2.7 of \citealt{Bezhanishvili2019}.\end{remark}

Dualizing Theorem \ref{FMS}, we obtain the following.

\begin{theorem}[Funayama for join semilattices]\label{FJS} For any poset $P$, there is an embedding $e$ of $P$ into a complete Boolean algebra such that:
\begin{enumerate}
    \item  $e$ preserves all existing finite joins;
\item  if $P$ is a join semilattice, then $e$ preserves all existing exact meets.
\end{enumerate}
\end{theorem}

\begin{proof} As in the proof of Theorem~\ref{FMS}, we may assume without loss of generality that $P$ is bounded. Let $P^\partial$ be the dual poset of $P$. Let $e^\partial$ be the embedding of $P^\partial$ into the complete Boolean algebra $\mathbf{B}(P^\partial)$ constructed in the proof of Theorem \ref{FMS}. Then $e^\partial$ (i) preserves finite meets and (ii) preserves all existing exact joins if $P^\partial$ is a meet semilattice. Therefore, $e^\partial$ is an embedding of $P$ into the dual $(\mathbf{B} (P{^\partial}))^\partial$ of $\mathbf{B}(P^\partial)$, which (a) preserves finite joins by (i) and (b) preserves all existing exact meets if $P$ is a join semilattice, in which case $P^\partial$ is a meet semilattice and (ii) applies.\end{proof}

\begin{remark}\label{Duals} The posets $X_{P}$ and $X_{P^\partial}$ are isomorphic via the map $f$ defined by $f(a,b)=(b,a)$. Note that $a\not\leq b$ iff $b\not\leq^\partial a$, so $f$ is a well-defined bijection from $X_{P}$ to $X_{P^\partial}$. To see that it is an order-isomorphism, observe that
\begin{eqnarray*}
(a,b)\sqsubseteq_{X_P}(c,d)&\Longleftrightarrow& a\leq_P c \mbox{ and }b\geq_P d\\
&\Longleftrightarrow& c\leq_{P^\partial} a \mbox{ and } d\geq_{P^\partial} b \\
&\Longleftrightarrow& (b,a)\sqsubseteq_{X_{P^\partial}} (d,c).
\end{eqnarray*}
It follows that $\mathbf{B}P$ is isomorphic to $\mathbf{B} (P{^\partial})$ and hence to $(\mathbf{B} (P{^\partial}))^\partial$ by the self-duality of Boolean algebras.\end{remark}

Despite Remark~\ref{Duals}, Theorems~\ref{FMS} and \ref{FJS} cannot be directly combined to obtain a single embedding from a lattice into a complete Boolean algebra that preserves both exact joins and exact meets, since we cannot assume that $f[e(a)]=e^\partial(a)$.\footnote{Indeed, let $P$ be the result of adding a new top above the 4-element Boolean algebra, so $P^\partial$ is the result of adding a new bottom below the 4-element Boolean algebra. Then where $a$ is one of the atoms of $P$, one can check that $f[e(a)]\neq e^\partial(a)$.} However, by assuming distributivity, we are able to prove that $e$ preserves all exact joins and exact meets.

\begin{theorem}[Funayama for distributive lattices]\label{FD} For any distributive lattice $L$, there is an embedding $e$ of $L$ into a complete Boolean algebra that preserves all  exact joins and exact meets.
\end{theorem}

\begin{proof} We continue from the end of the proof of Theorem \ref{FMS} but now assume that $P$ is a distributive lattice. To show that $e$ preserves exact meets, consider any $S\subseteq P$ that has an exact meet. First, suppose $\bigwedge S\neq 0$. Then to show that $e(\bigwedge S)= \bigsqcap \{e(s)\mid s\in S\}$, we must show that
\[\Box\Diamond {\Downarrow}\left(\bigwedge S,0\right) = \bigcap \{\Box\Diamond{\Downarrow}(s,0) \mid s\in S\}.\]
For the left-to-right inclusion, for each $s\in S$, since $\bigwedge S\leq s$, we have $(\bigwedge S, 0)\sqsubseteq (s,0)$ and hence $\Box\Diamond{\Downarrow}(\bigwedge S, 0)\subseteq  \Box\Diamond{\Downarrow}(s,0)$. Therefore, $\Box\Diamond {\Downarrow}(\bigwedge S,0) \subseteq \bigcap \{\Box\Diamond{\Downarrow}(s,0) \mid s\in S\}$.

For the right-to-left inclusion, assume $(x,y)\in \bigcap \{\Box\Diamond {\Downarrow} (s,0)\mid s\in S\}$. To show  $(x,y)\in \Box\Diamond {\Downarrow}(\bigwedge S,0)$, consider any $(x',y')\sqsubseteq (x,y)$.  For contradiction, suppose $(x',y')\not\in \Diamond {\Downarrow}(\bigwedge S,0)$, i.e., ${\Downarrow}(x',y')\cap {\Downarrow}(\bigwedge S,0)=\varnothing$. Then $(x'\wedge \bigwedge S, y')\not\in X_P$, so $x'\wedge \bigwedge S\leq y'$ and hence 
\begin{eqnarray*}
y'&=&y'\vee \left(x'\wedge \bigwedge S\right)\\
&=&(y'\vee x')\wedge \left(y'\vee \bigwedge S\right) \quad\mbox{ by distributivity}\\
&=& (y'\vee x')\wedge \bigwedge_{s\in S} (y'\vee s) \quad\mbox{ since }\bigwedge S\mbox{ is an exact meet}.
\end{eqnarray*}
Since $x'\not\leq y'$, it follows that $x'\not\leq \bigwedge_{s\in S} (y'\vee s)$, so there is some $s\in S$ such that $x'\not\leq y'\vee s$. Now $(x',y'\vee s)\sqsubseteq (x,y)$, which with $(x,y)\in \bigcap \{\Box\Diamond {\Downarrow} (s,0)\mid s\in S\}$ implies $(x',y'\vee s)\in \Diamond {\Downarrow}(s,0)$. Hence there is some $(x_s,y_s)\sqsubseteq (x',y'\vee s)$ such that $(x_s,y_s)\sqsubseteq (s,0)$. But then $x_s\leq x'\wedge s \leq y'\vee s\leq y_s$, a contradiction.

Next suppose $\bigwedge S=0$. Then $e(\bigwedge S)=e(0)=\varnothing\subseteq \bigsqcap\{e(s)\mid s\in S\}$. For the reverse inclusion, we claim that $\bigsqcap\{e(s)\mid s\in S\}=\varnothing$. If $0\in S$, this is immediate, since $e(0)=\varnothing$. So assume $0\not\in S$. Further suppose for contradiction that $(x,y)\in \bigsqcap\{e(s)\mid s\in S\}$, so $(x,y)\in \bigcap \{\Box\Diamond {\Downarrow} (s,0)\mid s\in S\}$. Since $\bigwedge S=0$, we have $x\wedge \bigwedge S\leq y$. Then we derive a contradiction exactly as above but replacing $x'$ by $x$ and $y'$ by $y$.\end{proof}

Whether the distributivity of $L$ in Theorem \ref{FD} is necessary is a question we leave open.

\begin{problem}\label{1stProblem} Is there a non-distributive lattice $L$ for which there is no embedding of $L$ into a complete Boolean algebra that preserves all finite meets, all exact joins, and all exact meets?
\end{problem}

We now obtain Funayama's theorem as an immediate corollary of Theorem \ref{FD}.

\begin{corollary}[Funayama] For any lattice $L$, there is a complete embedding of $L$ into a complete Boolean algebra iff $L$ satisfies both JID and MID.
\end{corollary}

\begin{proof}
As noted in the Introduction, the left-to-right direction is clear from von Neumann's \citeyearpar{Vonneumann1936} observation that complete Boolean algebras satisfy both JID and MID. For the right-to-left direction, if $L$ satisfies both JID and MID, then every existing join and meet is exact. Since $L$ is a lattice, all nonempty finite subsets have exact meets and joins. Thus, $L$ is distributive. Now apply \cref{FD}. 
\end{proof}

\section{Connection to MacNeille completion}

Given a lattice $L$, let $e:L\to\mathbf{B}L$ be the embedding from the proof of Theorem~\ref{FMS}.\footnote{We assume without loss of generality in this section that $L$ is bounded with $0\neq 1$; see the proof of Theorem~\ref{FMS} for justification.} Since $\mathbf{B}L$ is a Boolean algebra, the image $e[L]$ generates a Boolean subalgebra of $\mathbf{B}L$, which we denote by $\mathscr{B}L$.

\begin{remark}
We intentionally use the same notation, $\mathscr{B}L$, as we did for the Boolean envelope of a distributive lattice $L$ in Section~\ref{Intro}, since if $e$ is a lattice embedding of a distributive lattice $L$ into a Boolean algebra $B$, then the subalgebra of $B$ generated by $e[L]$ is isomorphic to the Boolean envelope of $L$ (see, e.g., \citealt[Sec.~V.4]{Balbes1974}).
\end{remark}

A natural question is how $\mathscr{B}L$ relates to $\mathbf{B}L$. Recall that the MacNeille completion of a Boolean algebra $B$ may be described, up to isomorphism, as a complete Boolean algebra $\mathfrak{M}B$ together with a Boolean embedding $h:B\to \mathfrak{M}B$ such that $h[B]$ is dense in $\mathfrak{M}B$, i.e., for every nonzero $x\in \mathfrak{M}B$, there is a nonzero $y\in h[B]$ such that $y\leq x$ (see \citealt[p.~153]{Sikorski1969}). Thus, if $B$ is a subalgebra of a complete Boolean algebra $C$, then $C$ is isomorphic to the MacNeille completion of $B$ just in case $B$ is a dense subalgebra of $C$.

\begin{proposition}\label{DLMacNeille} For any distributive lattice $L$, $\mathbf{B}L$ is isomorphic to the MacNeille completion of $\mathscr{B}L$.
\end{proposition}

\begin{proof} It suffices to show that $\mathscr{B}L$ is dense in $\mathbf{B}L$. Consider a nonempty $U\in \mathbf{B} L$, so there is some $(c,d)\in U$. Let $e$ be the embedding of $L$ into $\mathbf{B}L$ used in the proofs of Theorems \ref{FMS} and \ref{FD}. Since $(c,d)\in X_L$ and hence $c \not\leq d$, we have $e(c) \not\subseteq e(d)$,  so $e(c) \wedge \lnot e(d) \ne \varnothing$. We claim that $e(c)\wedge\neg e(d)\subseteq U$. Since $(c,d)\in U\in\mathbf{B}L$, we have $\Box\Diamond \mathord{\Downarrow} (c,d)\subseteq U$, so it suffices to show that $e(c)\wedge\neg e(d)\subseteq \Box\Diamond \mathord{\Downarrow} (c,d)$. Suppose $(w,v)\in e(c)\wedge\neg e(d)$. To show that $(w,v)\in\Box\Diamond\mathord{\Downarrow} (c,d)$, consider any $(w',v')\sqsubseteq (w,v)$, so $w'\leq w$ and $v\leq v'$. We must find a $(w'',v'')\sqsubseteq (w',v')$ such that $(w'',v'')\sqsubseteq (c,d)$. Since $(w,v)\in e(c)=\Box\Diamond\mathord{\Downarrow}(c,0)$, we have $(w',v')\in \Diamond\mathord{\Downarrow}(c,0)$, so there is some $(w'',v'')\sqsubseteq (w',v')$ such that $(w'',v'')\sqsubseteq (c,0)$. It follows that $w''\leq c\wedge w'$ and $w''\not\leq v''\geq v'$, so $w''\not\leq v'$ and hence $c\wedge w'\not\leq v'$. Since $(w,v)\in\neg e(d)=\neg\Box\Diamond\mathord{\Downarrow}(d,0)$, we have $d\wedge w\leq v$, since otherwise $(d\wedge w,v)\sqsubseteq (w,v)$ and $(d\wedge w,v)\in \mathord{\Downarrow}(d,0)\subseteq \Box\Diamond\mathord{\Downarrow}(d,0)$. Then given $d\wedge w\leq v$,  $w'\leq w$, and $v\leq v'$, we have $d\wedge w'\leq v'$. Toward showing that $c\wedge w'\not\leq d\vee v'$, suppose for contradiction that $c\wedge w'\leq d\vee v'$. Then $c\wedge w'=(c\wedge w')\wedge (d\vee v')=(c\wedge w'\wedge d)\vee (c\wedge w'\wedge v')$, using distributivity for the second equation. But then since $d\wedge w'\leq v'$ and $c\wedge w'\wedge v'\leq v'$, it follows that $c\wedge w'\leq v'$, contradicting our derivation of $c\wedge w'\not\leq v'$ above. Thus, $c \wedge w'\not\leq d\vee v'$, so $(c\wedge w',d\vee v')\in X_L$. Moreover, $(c\wedge w',d\vee v')\sqsubseteq (w',v')$ and $(c\wedge w',d\vee v')\sqsubseteq (c,d)$. Therefore, we may take $(w'',v'')=(c\wedge w',d\vee v')$. This completes the proof that $(w,v)\in\Box\Diamond \mathord{\Downarrow}(c,d)$ and hence that $e(c)\wedge\neg e(d)\subseteq \Box\Diamond \mathord{\Downarrow}(c,d)$, which shows that $\mathscr{B}L$ is dense in $\mathbf{B}L$.\end{proof}

It follows from Proposition \ref{DLMacNeille} that for distributive lattices, our construction is isomorphic to that of Gr\"atzer, and for complete Heyting algebras, it is isomorphic to that of Isbell (recall Section~\ref{Intro}). However, for non-distributive lattices, it is not guaranteed that $\mathbf{B}L$ is the MacNeille completion of $\mathscr{B}L$, as shown by the following example using the lattice $M_3$ displayed on the left of Figure~\ref{M3fig}. 

\begin{figure}[h]
\begin{center}
\begin{minipage}{1.5in}
\begin{center}
\begin{tikzpicture}[
    every node/.style={inner sep=2pt, font=\small},
    every edge/.style={draw}
]
\node (1) at (0,0) {$1$};
\node (x) at (-1,-1) {\textcolor{red}{$a$}};
\node (y) at (0,-1) {\textcolor{blue}{$b$}};
\node (z) at (1,-1) {\textcolor{medgreen}{$c$}};
\node (0) at (0,-2) {$0$};
\path (1) edge (y);
\path (1) edge (x);
\path (1) edge (z);
\path (x) edge (0);
\path (y) edge (0);
\path (z) edge (0);
\end{tikzpicture}
\end{center}
\end{minipage}
\begin{minipage}{3.25in}
\begin{center}
\begin{tikzpicture}[
    every node/.style={inner sep=2pt, font=\small},
    every edge/.style={draw, ->, >=Stealth}
]
 
\node (10) at (0, 0)          {$(1,0)$};
 
\node (1a) at (30:2)          {$(1,a)$};
\node (1b) at (270:2)         {$(1,b)$};
\node (1c) at (150:2)         {$(1,c)$};
\node (a0) at (210:2)         {\textcolor{red}{$(a,0)$}};
\node (b0) at (90:2)          {\textcolor{blue}{$(b,0)$}};
\node (c0) at (330:2)         {\textcolor{medgreen}{$(c,0)$}};
 
\node (ba) at ( 1.732,  3)    {\textcolor{blue}{$(b,a)$}};
\node (bc) at (-1.732,  3)    {\textcolor{blue}{$(b,c)$}};
\node (ac) at (-3.464,  0)    {\textcolor{red}{$(a,c)$}};
\node (ab) at (-1.732, -3)    {\textcolor{red}{$(a,b)$}};
\node (cb) at ( 1.732, -3)    {\textcolor{medgreen}{$(c,b)$}};
\node (ca) at ( 3.464,  0)    {\textcolor{medgreen}{$(c,a)$}};
 
\path (10) edge (1a);
\path (10) edge (1b);
\path (10) edge (1c);
\path (10) edge (a0);
\path (10) edge (b0);
\path (10) edge (c0);
\path (1a) edge (ba);
\path (1a) edge (ca);
\path (1b) edge (ab);
\path (1b) edge (cb);
\path (1c) edge (ac);
\path (1c) edge (bc);
\path (a0) edge (ab);
\path (a0) edge (ac);
\path (b0) edge (ba);
\path (b0) edge (bc);
\path (c0) edge (ca);
\path (c0) edge (cb);
\end{tikzpicture}
\end{center}
\end{minipage}
\end{center}
\caption{the lattice $M_3$ (left) and $X_{M_3}$ (right). An arrow from $(x,y)$ to $(x',y')$ indicates that $(x',y')\sqsubseteq (x,y)$. The color of the nodes indicates the embedding $e: M_3\to \mathbf{B}M_3$.}\label{M3fig}
\end{figure}
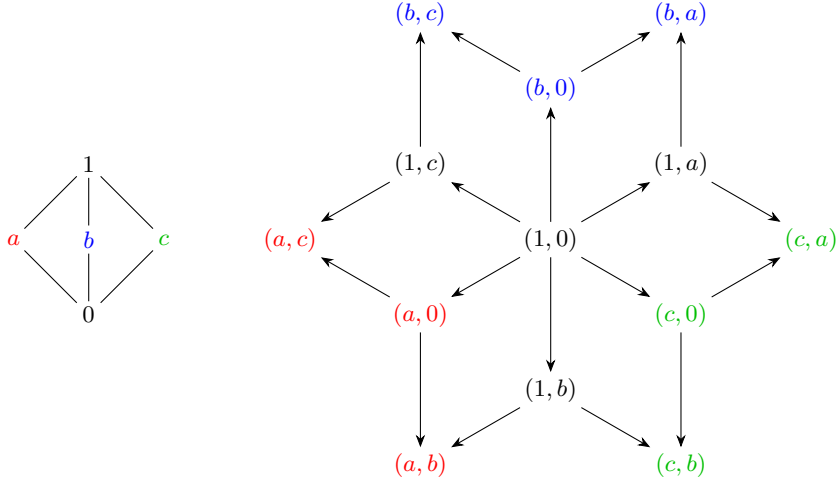

\begin{fact}\label{M3Fact} $\mathscr{B}M_3$ is a proper subalgebra of $\mathbf{B} M_3$.
\end{fact}

\begin{proof} We show that $\{(a,b)\} \in \mathbf{B} M_3 \setminus \mathscr{B}M_3$. First, we show that $\{(a,b)\} \in \mathbf{B} M_3$. For this, it suffices to show that $\{(a,b)\}=\Box\Diamond\mathord{\Downarrow}(a,b)$. To see this, consider Figure \ref{M3fig}, which shows both $M_3$ (left) and $X_{M_3}$ (right). Observe that $\{(a,b)\}$ is a $\sqsubseteq$-downset. To see that it is a $\Box\Diamond$-fixpoint, consider any $(x,y)\not\in \{(a,b)\}$, so $x\neq a$ or $y\neq b$. We will show that in either case there is an $(x',y')\sqsubseteq (x,y)$ such that for all $(x'',y'')\sqsubseteq (x',y')$, we have $(x'',y'')\not\in \{(a,b)\}$. If $x\neq a$, then either $x=1$ or $x\in \{b,c\}$. If $x=1$, then let $(x',y')=(b,y)$ if $y\neq b$ and $(x',y')=(c,b)$ otherwise.  On the other hand, if $x\in \{b,c\}$, then let $(x',y')=(x,y)$. Now assume $x=a$, so $y\neq b$. Then $y=0$ or $y=c$. In either case, let $(x',y')=(a,c)$.

Towards showing that $\{(a,b)\}\not\in \mathscr{B}M_3$, we first claim that $e(a) = \{(a, 0), (a, b), (a, c)\}$, where $e$ is the embedding of $M_3$ into $\mathbf{B}M_3$ used in the proofs of Theorems \ref{FMS} and \ref{FD}, so $e(a)=\Box\Diamond\mathord{\Downarrow}(a,0)$. Observe that $\mathord{\Downarrow}(a,0)= \{(a, 0), (a, b), (a, c)\}$. We claim that $\mathord{\Downarrow}(a,0)$ is a $\Box\Diamond$-fixpoint. Consider any ${(x,y)\not\in \mathord{\Downarrow}(a,0)}$, so $x=1$ or $x\in \{b,c\}$. Then by reasoning analogous to that in the previous paragraph, there is an $(x',y')\sqsubseteq (x,y)$ such that for all $(x'',y'')\sqsubseteq (x',y')$, we have $(x'',y'')\not\in \mathord{\Downarrow}(a,0)$. 

Next, we show that $\neg e(a)=\{(b, 0), (b, a), (b, c), (c, 0), (c, a), (c, b), (1, a)\}$. Call the right-hand side $Y$. Clearly, for all $(x,y)\in Y$, $\mathord{\Downarrow}(x,y)\cap e(a)=\varnothing$, so $(x,y)\in \neg e(a)$. Also observe that \[X_{M_3}\setminus Y=\{(a, 0), (a, b), (a, c),(1,b), (1,c), (1,0)\},\] and for every $(x,y)\in X_{M_3}\setminus Y$, there is some $(x',y')\in e(a)$ with $(x',y')\sqsubseteq (x,y)$, so $(X_{M_3}\setminus Y)\cap \neg e(a)=\varnothing$. This completes the proof that $\neg e(a)=Y$.

By symmetric reasoning for $b$ and $c$, we can compute the $e$-images of all elements of $M_3\setminus \{0,1\}$ and their negations:

\begin{center}
\begin{minipage}{2.25in}
\begin{itemize}
    \item[] $e(a) = \{(a, 0), (a, b), (a, c)\}$;
    \item[] $e(b) = \{(b,0), (b,a), (b,c)\}$;
    \item[] $e(c) = \{(c, 0),(c, a),(c, b)\}$;
    \end{itemize}\end{minipage}\begin{minipage}{3.75in}\begin{itemize}
    \item[] $\neg e(a)=\{(b, 0), (b, a), (b, c), (c, 0), (c, a), (c, b), (1, a)\}$;
    \item[] $\neg e(b) = \{(a, 0), (a, b), (a, c), (c, 0), (c, a), (c, b), (1, b)\}$;
    \item[] $\neg e(c) = \{(a, 0), (a, b), (a, c), (b, 0), (b, a), (b, c), (1, c)\}$.
\end{itemize}
\end{minipage}
\end{center}
\noindent Observe that $e(a)\subseteq \neg e(b)\cap \neg e(c)$, $e(b)\subseteq \neg e(a)\cap \neg e(c)$, $e(c)\subseteq \neg e(a)\cap \neg e(b)$, and $\neg e(a)\cap \neg e(b)\cap \neg e(c) =\varnothing$. Since every element of $\mathscr{B}M_3$ is a join of meets, where each meet includes either $e(x)$ or $\neg e(x)$ for $x\in \{a,b,c\}$, each atom of $\mathscr{B}M_3$ is such a meet. It follows from the inclusions noted before that $e(a)$, $e(b)$, and $e(c)$ are the only atoms.\footnote{Thus, $\mathscr{B}M_3$ is the 8-element Boolean algebra consisting of the $6$ elements in the bulleted list plus $\varnothing$ and $X_{M_3}$.} Then since $\{(a,b)\}$ is a proper subset of $e(a)$, it follows that $\{(a,b)\}\not\in \mathscr{B}M_3$.\end{proof}

\begin{remark} $\mathbf{B}M_3$ has 6 atoms, namely the six points of the ``snowflake'' in Figure~\ref{M3fig}: $\{(a,b)\}$, $\{(a,c)\}$, $\{(b,a)\}$, $\{(b,c)\}$, $\{(c,a)\}$, and $\{(c,b)\}$. Thus,  $\mathbf{B}M_3$ has $2^6=64$ elements.\end{remark}

Since the MacNeille completion of a finite Boolean algebra is (isomorphic to) itself, as an immediate consequence of Fact \ref{M3Fact} we obtain the following.

\begin{proposition}\label{NotMacneille} $\mathbf{B}M_3$ is not isomorphic to the MacNeille completion of $\mathscr{B}M_3$.\end{proposition}

 On the other hand, there are some non-distributive lattices $L$ such that $\mathbf{B}L$ is isomorphic to the MacNeille completion of $\mathscr{B}L$. An example is $N_5$, shown on the left of Figure~\ref{N5fig}, for which $\mathbf{B}N_5$ is generated by $e[N_5]$.

 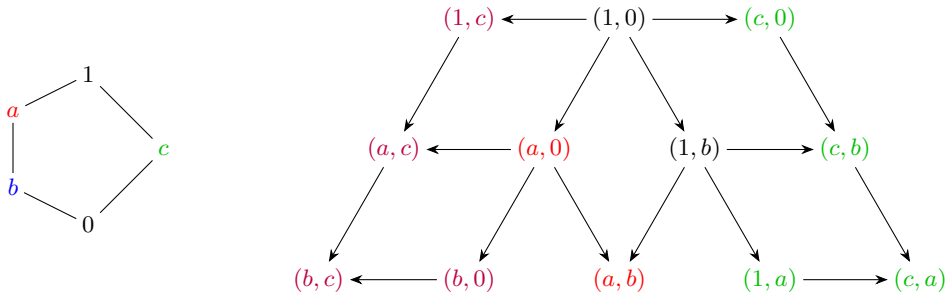
\begin{figure}[h]
\begin{center}
\begin{minipage}{1.5in}
\begin{tikzpicture}[
    every node/.style={inner sep=2pt, font=\small},
    every edge/.style={draw, ->, >=Stealth}
]

\node (n1) at (4,0) {{$1$}};
\node (nx) at (3,-.5) {{\textcolor{red}{$a$}}};
\node (ny) at (3,-1.5) {{\textcolor{blue}{$b$}}};
\node (nz) at (5,-1) {{\textcolor{medgreen}{$c$}}};
\node (n0) at (4,-2) {{$0$}};

\path (nx) edge[-] node {{}} (n1);
\path (nx) edge[-] node {{}} (ny);
\path (n1) edge[-] node {{}} (nz);
\path (ny) edge[-] node {{}} (n0);
\path (nz) edge[-] node {{}} (n0);

\end{tikzpicture}\end{minipage}\begin{minipage}{3.25in}\begin{tikzpicture}[
    every node/.style={inner sep=2pt, font=\small},
    every edge/.style={draw, ->, >=Stealth}
]
\node (10) at (0,0) {$(1,0)$};
\node (1a) at (2,-3.464) {\textcolor{medgreen}{$(1,a)$}};
\node (1b) at (-60:2)  {$(1,b)$};
\node (1c) at (180:2) {\textcolor{purple}{$(1,c)$}};
\node (a0) at (-120:2) {\textcolor{red}{$(a,0)$}};
\node (b0) at (-2,-3.464)  {\textcolor{purple}{$(b,0)$}};
\node (c0) at (0:2)   {\textcolor{medgreen}{$(c,0)$}};
\node (ab) at (0, -3.464)    {\textcolor{red}{$(a,b)$}};
\node (ac) at (-3, -1.732)  {\textcolor{purple}{$(a,c)$}};
\node (bc) at (-4,-3.464)  {\textcolor{purple}{$(b,c)$}};
\node (ca) at (4,-3.464)   {\textcolor{medgreen}{$(c,a)$}};
\node (cb) at (3, -1.732)   {\textcolor{medgreen}{$(c,b)$}};
 
\path (1a) edge (ca);
 
\path (1b) edge (1a);
\path (1b) edge (cb);
\path (1b) edge (ab);
 
\path (1c) edge (ac);
 
\path (10) edge (1b);
\path (10) edge (1c);
\path (10) edge (a0);
\path (10) edge (c0);
 
\path (a0) edge (b0);
\path (a0) edge (ab);
\path (a0) edge (ac);
 
\path (ac) edge (bc);
 
\path (b0) edge (bc);
 
\path (cb) edge (ca);
 
\path (c0) edge (cb);
\end{tikzpicture}

\end{minipage}
\end{center}
\caption{the lattice $N_5$ (left) and $X_{N_5}$ (right). An arrow from $(x,y)$ to $(x',y')$ indicates that $(x',y')\sqsubseteq (x,y)$. The color of the nodes indicates the embedding $e:N_5\to \mathbf{B}N_5$, where purple pairs are in the $e$-image of both $a$ and $b$.}\label{N5fig}
\end{figure}

\begin{fact}
$\mathbf{B} N_5=\mathscr{B} N_5$.
\end{fact}

\begin{proof} 
The poset $X_{N_5}$ is shown on the right of Figure~\ref{N5fig}. Direct calculation shows that each of the following sets is a $\Box\Diamond$-fixpoint of $\mathcal{D}(X_{N_5})$:

\begin{center}
\begin{minipage}{2.75in}
\begin{enumerate}
    \item[] $S_1 = \varnothing$;
    \item[] $S_2 = \{(a, b)\}$;
    \item[] $S_3 = \{(a, c), (b, 0), (b, c), (1, c)\}$;
    \item[] $S_4 = \{(c, a), (c, 0), (c, b), (1, a)\}$;
\end{enumerate}    
\end{minipage}\begin{minipage}{4in}
    \begin{enumerate}
    \item[] $S_5 = \{(a, b), (c, 0), (c, a), (c, b), (1, a), (1, b)\}$;
    \item[] $S_6 = \{(a, 0), (a, b), (a, c), (b, 0), (b, c), (1, c)\}$;
    \item[] $S_7 = \{(a, c), (b, 0), (b, c), (c, 0), (c, a), (c, b), (1, a), (1, c)\}$;
    \item[] $S_8 = X_{N_5}$.
    \end{enumerate}
\end{minipage}
\end{center}

\noindent Where $e$ is the embedding of $N_5$ into $\mathbf{B}N_5$ used in the proofs of Theorems \ref{FMS} and \ref{FD}, we have \[e(a)=S_6, \ e(b)=S_3, \ \mbox{and }e(c)=S_4.\] We claim that there are no other $\Box\Diamond$-fixpoints besides $S_1,\dots,S_8$. Each $\Box\Diamond$-fixpoint must be a $\sqsubseteq$-downset, so either it is $\varnothing$ or it must contain at least one of the minimal points, which are $(b,c)$, $(a,b)$, and $(c,a)$. Considering the minimal points from left to right, we reason as follows:
\begin{enumerate}
\item[] \hypertarget{case1}{Case 1}: $\mathrm{min}U=\{(b,c)\}$. Given that $(b,c)\in U$, we have $S_3\subseteq U$, since for every $(x,y)\in S_3$ and $(x',y')\sqsubseteq (x,y)$, we have $(b,c)\sqsubseteq (x',y')$. Moreover, given that $(a,b)\not\in U$ and $(c,a)\not\in U$, we have $U\subseteq S_3$, since for any $(x,y)\in X_{N_5}\setminus S_3$, we have $(a,b)\sqsubseteq (x,y)$ or $(c,a)\sqsubseteq (x,y)$. Therefore, $U=S_3$.
\item[] Case 2: $\mathrm{min}U=\{(a,b)\}$. Then $U=S_2$, since for any $(x,y)\in X_{N_5}\setminus S_2$, we have either $(b,c)\sqsubseteq (x,y)$ or $(c,a)\sqsubseteq (x,y)$.
\item[] \hypertarget{case3}{Case 3}: $\mathrm{min}U=\{(c,a)\}$. Then by reasoning symmetric to that in \hyperlink{case1}{Case 1}, $U=S_4$.

\item[] \hypertarget{case4}{Case 4}: $\mathrm{min}U=\{(b,c),(a,b)\}$. Given that $(b,c)\in U$ and $(a,b)\in U$, we have $(a,0)\in U$, since for any $(x',y')\sqsubseteq (a,0)$, we have either $(b,c)\sqsubseteq (x',y')$ or $(a,b)\sqsubseteq (x',y')$; it follows that $S_6\subseteq U$. Moreover, given that $(c,a)\not\in U$, we have $U\subseteq S_6$, since for any $(x,y)\in X_{N_5}\setminus S_6$, we have $(c,a)\sqsubseteq (x,y)$. Therefore, $U=S_6$.
\item[] Case 5: $\mathrm{min}U=\{(b,c),(c,a)\}$. Then $S_7=S_3\cup S_4\subseteq U$, by reasoning as in \hyperlink{case1}{Case 1} and \hyperlink{case3}{Case 3}. Moreover, $U\subseteq S_7$, since for every $(x,y)\in X_{N_5}\setminus S_7$, we have $(a,b)\sqsubseteq (x,y)$. Therefore, $U=S_7$.
\item[] Case 6: $\mathrm{min}U=\{(a,b),(c,a)\}$. Then by reasoning symmetric to that in \hyperlink{case4}{Case 4},  $U=S_5$.
\item[] Case 7: $\mathrm{min}U=\{(b,c),(a,b),(c,a)\}$. Then clearly $U=X_{N_5}$.
\end{enumerate}

When ordered by inclusion, $S_1,\dots,S_8$ form the $8$-element Boolean algebra $\mathbf{B} N_5$. Then since $e[N_5]$ has 5 elements, the Boolean subalgebra $\mathscr{B}N_5$ of $\mathbf{B}N_5$ generated by $e[N_5]$ must be $\mathbf{B} N_5$ itself.\end{proof}

As an immediate consequence, we obtain the following.

\begin{proposition}$\mathbf{B} N_5$ is isomorphic to the MacNeille completion of $\mathscr{B} N_5$.
\end{proposition}

The contrast between the foregoing propositions suggests the following problem.

\begin{problem}\label{2ndProblem} Characterize the class of lattices $L$ such that $\mathbf{B}L$ is isomorphic to the MacNeille completion of $\mathscr{B}L$.
\end{problem}

\subsection*{Acknowledgements}

We thank Mamuka Jibladze for helpful discussions that led to this paper.

\bibliographystyle{plainnat}
\bibliography{funayama}

\end{document}